\documentclass[noams]{compositio}
\usepackage{lineno,hyperref,longtable, multirow,graphicx,color,appendix}

\hypersetup{
	colorlinks = true,
	linkcolor = blue,
	anchorcolor = blue,
	citecolor = blue,
	filecolor = blue,
	urlcolor = blue
}
\usepackage{epsfig}  		
\usepackage{epic,eepic}
 \usepackage{color}
 \usepackage[all,arc]{xy}
\usepackage{mathrsfs}
\usepackage[normalem]{ulem}
\usepackage{amsmath,amssymb,enumerate,shellesc}
\newtheorem{mainthm}{Theorem}

\newtheorem{theorem}{Theorem}[section]
\newtheorem{proposition}[theorem]{Proposition}
\newtheorem{lemma}[theorem]{Lemma}
\newtheorem{corollary}[theorem]{Corollary}

\theoremstyle{definition}

\theoremstyle{remark}

\newtheorem{remark}[theorem]{Remark}

\newcommand{\R}{\mathbb{R}}  

\newcommand{\C}{\mathbb{C}}
\newcommand{\N}{\mathbb{N}}
\newcommand{\Z}{\mathbb{Z}}
\newcommand{\fl}{f_{\lambda}}
\newcommand{\fz}{\{f_{\lambda}^n(z)\}_{n=1}^{\infty}}
\newcommand{\f}{\{f_{\lambda}^n\}_{n=1}^{\infty}}
\newcommand{\J}{\mathcal{J}(\{f_{\lambda}^n\}_{n=1}^{\infty})}
\newcommand{\I}{\mathcal{I}(\{f_{\lambda}^n\}_{n=1}^{\infty})}
\newcommand{\F}{\mathcal{F}(\{f_{\lambda}^n\}_{n=1}^{\infty})}

\renewcommand{\contentsname}{Contents\\{\footnotesize\normalfont(A table
of contents should normally not be included)}}
\begin{document}
\title[non-autonomous iterations of exponential maps]{On the Julia sets of the non-autonomous iteration of exponential maps }

\author{Jiaxing Huang}
\email{hjxmath@szu.edu.cn}
\address{School of Mathematical Sciences, Shenzhen University, 518061, P. R. China}

\author{Xiao Yao}
\email{yaoxiao@nankai.edu.cn}
\address{School of Mathematical Sciences, Nankai University, P.R. China}
\thanks{The first author was supported by NSFC of China (No. 12201420 and 12231013) and the NSF of Guangdong Province (No. 2026A1515010851), and the second author was supported by NSFC of China (No.12371074).}

\begin{abstract}
This paper studies non‑autonomous iterations of $\lambda e^z$, where at each step the parameter $\lambda$ is chosen from a bounded interval. Urbański and Zdunik  introduced and developed a thermodynamic formalism for random dynamics of exponential maps, extending the autonomous result established independently by Rees and Lyubich to typical parameter sequences in $[A, B]^{\mathbb{N}}$ with $1/e<A<B<\infty$.

In this paper, we show that the same conclusion indeed holds for all parameters in $[A, B]^{\mathbb{N}}$, Our proof is based on carefully crafted uniform versions of pullback arguments, inspired by the original methods of Rees and Lyubich.
In a related direction, Urbański and Zdunik established a non‑autonomous analogue of Misiurewicz’s theorem for the exponential map, later refined by Lech. We further extend their work by providing a refined condition for the Julia set to be the entire complex plane. 
\end{abstract}

\maketitle
\section{Introduction and main results}
The dynamics of the exponential family $E_{\lambda}(z)=\lambda e^z,\lambda\in\C\setminus\{0\}$ occupies a central position in the development of transcendental dynamics. Despite its deceptively simple form, this family exhibits remarkably rich and intricate dynamical behaviour, both from the perspective of topological dynamics and ergodic theory. Significant progress in understanding its properties has been made over the past four decades, see, for instance, \cite{Misiu, BaR84, DK84, Lyu87, Rees86, RS09, Dob15, AC24} and reference therein.

When one turns to non‑autonomous transcendental dynamics, the iteration of exponential maps arises naturally as the “simplest” family in this setting. Indeed, Urbański, Zdunik \cite{UZ21}, and Lech \cite{Lez21} have extended Misiurewicz’s classical theorem \cite{Misiu} on the autonomous exponential map to the non‑autonomous case. Moreover, Urbański and Zdunik \cite{UZ21} developed a systematic thermodynamical formalism for random iterations of exponential maps, building upon a series of their foundational works on classical exponential dynamics \cite{UZ03, UZ04, UZ04b, UZ07, UZ21}. Barański and Karpińska \cite{BK} systematically studied the fractal dimension of the non‑autonomous escaping set for the exponential family with a given escape rate.

Nevertheless,  due to the absence of standard notions such as (repelling) periodic orbits, recurrence, and ergodicity in the non‑autonomous framework, the overall global dynamical picture of the non‑autonomous exponential family remains far from clear.
It is widely believed, however, that the non-autonomous dynamics of the exponential family remains highly complex and rich in some sense. Key concepts from the autonomous iteration theory, such as the structure of dynamic rays and the classification of escaping points, are expected to admit suitable analogues in this setting. Some initial steps in this direction have already been anticipated in the work of Rottenfusser el at \cite{RRRS11}.

In this paper, we study non-autonomous iterates of exponential maps $f_{\lambda}(z)=\lambda e^z$ where at each step the parameter $\lambda$ is  chosen from a bounded real interval $[A, B]\subset\R$, $0<A<B<\infty$.
Such a non-autonomous iteration can be described as a skew product dynamical system.  
Let $\Lambda$ be a  compact set in $\mathbb{C}$ and $X=\Lambda^{\mathbb{N}}$.  Define the shift map $\sigma$ from  $X$ to itself by
$$
\sigma((\lambda_{n})_{n=1}^{\infty})=
(\lambda_{n+1})_{n=1}^{\infty}. 
$$
The skew product dynamical system is defined by
$F: X\times \mathbb{C}\mapsto X\times \mathbb{C}$,
$$
F(\lambda, z)=(\sigma \lambda, f_{\lambda_1}(z)).
$$
Thus, we have
$$
F^m(\lambda, z)=(\sigma^{m}\lambda, f_{\lambda}^m(z)),
$$
where $f_{\lambda}^{m}=f_{\lambda_m}\circ\cdots \circ f_{\lambda_1}$. So, for each sequence $\lambda\in X$, one can determine a sequence of non-automonous iterates $\{f_{\lambda}^n\}_{n\in\N}$.
Analogously to the autonomous case, the non-autonomous Fatou set $\F$ is the set of all $z\in\C$ for which $\{f_{\lambda}^n\}_{n=1}^{\infty}$ is normal on a neighbourhood of $z$, and the non-autonomous Julia set $\J$ is simply the complement of the Fatou set $\F$.

In 1981, Misiurewicz \cite{Misiu} proved that the Julia set of the exponential map $\exp(z)$ is the entire complex plane. This foundational result inspired extensive work on its measurable dynamics. For example, Rees \cite{Rees86} and Lyubich \cite{Lyu87} independently showed that for $\exp(z)$, almost every forward orbit in $\C$ accumulates only on the post‑singular set. Later, Urbański and Zdunik \cite{UZ07,UZ21} and Hemke \cite{Hem} generalized this to broader classes of exponential maps, including those with a super‑growing singular value (in the sense of iterations) and random iterations of non‑hyperbolic exponentials. For related results in the setting of general meromorphic functions, we refer the reader to \cite{KK97, KK99, Sko03, Sko05, KU06, CJK24}.

In the autonomous setting, such results are part of a larger dynamical paradigm. For meromorphic functions, especially those of finite type, the asymptotic behaviour of orbits in the Julia set is closely related to recurrent and ergodic properties, and the absence of invariant line fields (see e.g. \cite{RVS11}). A fundamental dichotomy, originating in the work of Lyubich \cite{Lyu83} and McMullen \cite[Theorem 3.9]{McM91} for rational maps of degree at least 2, illustrates this link: either the Julia set is the whole Riemann sphere and the map acts ergodically, or almost every orbit of Julia sets is attracted by the post‑critical (or, more generally, post‑singular) set. This dichotomy was later extended by Bock \cite{Bock} to general meromorphic functions. Thus, in the autonomous case, describing the limit sets of forward orbits in Julia sets, somehow, is central to understanding whether the system exhibits global ergodicity and recurrence.

Therefore, it is natural to ask whether an analogue of the Lyubich–McMullen dichotomy holds in the non‑autonomous case.  However, due to the absence of forward and backward properties of Fatou and Julia sets, together with the lack of periodicity in the Fatou set, the classical ergodic and recurrence approach seems inapplicable to the non-autonomous framework. For this reason, the present paper focuses on the corresponding structural questions in the non‑autonomous framework,  while the issues of ergodicity and recurrence will be addressed in future work. Specifically, we address:
\begin{enumerate}[(I)]
\item \emph{When does the Julia set fill the whole plane}?
\item \emph{What can be said about the limit sets of forward orbits on the Julia set}?
\end{enumerate}
We aim to establish counterparts of the above autonomous dichotomy for sequences of exponential maps, thereby extending the earlier results of Rees, Lyubich, Urbański–Zdunik and Hemke to non‑autonomous systems. In particular, we identify conditions under which the Julia set is the whole plane, and describe the typical orbits in the Julia set.

Now, we would like to concern the first question. Some earlier work gives partial answers:


\begin{mainthm}[\cite{UZ21}, Theorem 7]
Let $\lambda$ be a sequence in $[A, B]^{\N}, A>1/e$. Then the Julia set $$\J=\C.$$
\end{mainthm}
Lech later generalized this by considering sequences with terms possibly equal to $1/e$.
\begin{mainthm}[\cite{Lez21}, Theorem 3.1]\label{thm-Lez}
Let $\lambda=(\lambda_n)_{n\in\N}$ be a sequence in $([A, B]\cup\{\frac{1}{e}\})^{\N}$, where $1/e<A<B<\infty$. Then the Julia set $\J$ is the whole complex plane if and only if $\lambda_n>A$ for infinitely many $n$.
\end{mainthm}

While both results require the sequence to be bounded away from $1/e$ asymptotically, Lech’s theorem allows individual terms to equal $1/e$ as long as infinitely many terms exceed a fixed lower bound $A>1/e$. In this sense, Lech’s statement is a direct generalization of Urbański–Zdunik’s, and their proofs follow similar autonomous-like techniques applicable when the parameters stay eventually above 
$1/e$.

Our first theorem extends the results of Urba\'nski–Zdunik and Lech by allowing subsequences of $(\lambda_n)$ to converge to the critical parameter $1/e$. It only requires the limsup of the sequence to exceed $1/e$, thereby covering the previously unaddressed case where the parameter sequence can accumulate at the critical parameter $1/e$.
\begin{theorem}\label{thm-whole-plane}
Let $\lambda=(\lambda_n)_{n=0}^{\infty}$ be a sequence such that 
$\lambda_n\in [\frac{1}{e}, M]$ for some constant $M>\frac{1}{e}$ and
$\limsup\limits_{n\to\infty}\lambda_n>\frac{1}{e}$,   then we have $$\J=
\C.$$
\end{theorem}

Given a sequence  $\lambda=(\lambda_{n})_{n=1}^{\infty}$ and $z\in \mathbb{C}$, we define
\begin{equation*}\label{eq-omega}
\omega_{\lambda}(z):=\{\text{all possible  accumulating points of} \
\fz\ \text{in} \ \hat{\mathbb{C}}\}.
\end{equation*}
The second line of research concerns the measurable dynamics of non‑autonomous exponential iterations. 
After the work independently obtained by Rees \cite{Rees86} and Lyubich \cite{Lyu87} for dynamics of the exponential map $e^z$, Urbański and Zdunik established the following random analogue, which can be seen as resonance of their works (see \cite{UZ07, UZ03}).
\begin{mainthm}[\cite{UZ21}, Corollary 6]
Let $\mu$ be a probability measure on $[A, B]^{\mathbb{N}}$, where $1/e<A<B<\infty$. Then for $\mu$-almost all sequence $\lambda=(\lambda_n)_{n=1}^{\infty}$, there exists a Borel set $Q_{\lambda}$ with full $2$-dimensional Lebesgue measure such that for any $z\in Q_{\lambda}$, the set $\omega_{\lambda}(z)$ is contained in $[0, +\infty]\cup\{-\infty\}$ and contains $+\infty$.
\end{mainthm}

Motivated by the fact that random dynamical systems lie naturally between autonomous and non‑autonomous iterations, offering a unified framework for extending classical results. 
we remove the randomness assumption and prove a purely deterministic non‑autonomous counterpart.

Our second main result is the following.

\begin{theorem}\label{thm-main}
Let $\lambda=(\lambda_n)_{n=1}^{\infty} $ be a bounded sequence in $[A, B]^{\mathbb{N}}$, where $\frac{1}{e}<A<B<\infty$, then for almost all $z$ in $\mathbb{C}$, 
$$
\omega_{\lambda}(z)\subset\mathcal{B}_{\lambda}, 
$$
where 
\begin{equation}\label{eq-B}
\mathcal{B}_{\lambda}=\overline{\bigcup_{m=0}^{\infty}
\bigcup_{k=1}^{\infty}\left\{f_{\sigma^m \lambda}^{k}(0)\right\}}\bigcup \{0, \infty\}.
\end{equation}
\end{theorem}
Using the invariance $\omega_{\sigma^k\lambda}(f_{\lambda}^k(z))=\omega_{\lambda}(z)$, we can strengthen Theorem \ref{thm-main} as follows.
\begin{corollary}\label{cor-main}
For almost all $z$ in $\mathbb{C}$, 
\begin{equation*}
\omega_{\lambda}(z)\subset \bigcap_{k=0}^{\infty}\mathcal{B}_{\sigma^k \lambda}. 
\end{equation*}
\end{corollary}
\begin{proof}
For any $k\in \mathbb{N}$, there exists a set $E_k\in \mathbb{C}$ of full Lebesgue measure such that
$$
\omega_{\sigma^k\lambda}(w)\subset \mathcal{B}_{\sigma^k\lambda}\quad\text{for every}\ w\in E_{k}.
$$

Since any non‑constant analytic function preserves sets of zero Lebesgue measure, the pre-image $B_k:=(f_{\lambda}^{k})^{-1}(E_k)$ also has full measure in $\C$.
Thus, for any $z\in B_k$,
$$
\omega_{\sigma^k\lambda}(f_{\lambda}^{k}(z))\subset \mathcal{B}_{\sigma^k\lambda}
$$
Recalling that $\omega_{\sigma^k\lambda}(f_{\lambda}^k(z))=\omega_{\lambda}(z)$, we obtain that 
$$
\omega_{\lambda}(z)\subset \mathcal{B}_{\sigma^k\lambda},\quad\text{for all}\ z\in B_k.
$$
Finally, taking the intersection over all $k\in\N$ gives
\[
\omega_{\lambda}(z)\subset \bigcap_{k=1}^{\infty}\mathcal{B}_{\sigma^k\lambda}
\]
for all $z\in \displaystyle\bigcap_{k=1}^{\infty}B_k$, where the intersection set $ \displaystyle\bigcap_{k=1}^{\infty}B_k$ still has full Lebesgue measure.

\end{proof}

\begin{remark}
The condition $\frac{1}{e}<A<B<\infty$ can, in a certain sense, be relaxed. The method employed in this paper remains valid under the more general assumption that $\lambda_n\in [\frac{1}{e}, \infty] \ (n\in \mathbb{N})$
 for all $n\in \mathbb{N}$, provided that
\begin{equation}\label{eq-refine}
\lim_{j\rightarrow\infty}\inf_{k}f_{\sigma^k\lambda}^{j}(0)=+\infty.
\end{equation}
One may note that condition \eqref{eq-refine} is equivalent to
\begin{equation*}
\lim_{j\rightarrow\infty}\inf_{k}f_{\sigma^k\lambda}^{j}(0)>1.
\end{equation*}
It is therefore natural to investigate whether the non-autonomous analog of Rees and Lyubich’s theorem remains valid under the condition on the sequence $\lambda$ described in Theorem \ref{thm-whole-plane} or Theorem \ref{thm-Lez}.
\end{remark}


\textbf{Structure of the article.} In Section \ref{sec:Julia}, we show that for non‑autonomous iterates of functions from the Eremenko–Lyubich class, every escaping point is a Julia point, and use this fact to prove Theorem \ref{thm-whole-plane}.

Section \ref{sec:covering} presents several covering properties of non‑autonomous iterations of exponential maps, including a blow‑up property of the Julia set and a contraction property of backward orbits. These results play essential roles in the proof of our main theorem (Theorem \ref{thm-main}).
Section \ref{sec:orbit} is devoted to a key sufficient condition (Theorem \ref{thm-key}) for Theorem \ref{thm-main}.

Follow the similar ideas from the work of Rees \cite{Rees86} and Lyubich \cite{Lyu87},
Section \ref{sec:loss} establishes a geometric framework for quantitatively estimating the area of “non-typical” point sets, specifically type 0 points (Proposition \ref{pro-type0}) and $\epsilon$-grid loss points (Proposition \ref{pro-grid}), in non‑autonomous exponential iterations. These results provide essential tools for studying the typical orbits in Julia sets.  
All uniform estimates for the area of non‑typical points, exhibiting exponential decay in the real‑part parameter, are independent of the choice of the parameter. 
This essentially depends on the strong expansion along orbits in the right half‑plane with large real parts. In our setting, 
such uniform estimates are crucial and required significantly more effort than in previous works, 
particularly in terms of the corresponding distortion estimates and the combinatorial description of possible non‑typical points.

Finally, some elementary properties of exponential iterations at the origin are presented in the appendix Section.

\textbf{Some notations.} Let $\C$ be the complex plane and Z the set of integers. The real
and imaginary parts of $z\in\C$ are denoted by $\Re(z)$ and $\Im(z)$, respectively. 
We shall use $D(z, r)$ for the disk of radius $r$ centered at $z$, and $\Delta$ for the unit disk $D(0,1)$.
Furthermore, by $S(z, r)$ we mean
the open square of side-length $2r$ centered at $z$: $$S(z, r):=\{z\in\C: |u-\Re z|<r, |y-\Im z|<r\},$$
and $kS(z,r)=S(z,kr),k>0$.

\paragraph{\bf Acknowledgment.}
This work was supported by the Visiting Scholar Program
of Chern Institute of Mathematics at Nankai University. The first author thanks the Chern Institute of Mathematics for its hospitality during his visit.  

\section{Julia sets for non-autonomous exponential maps}\label{sec:Julia}
In this section, we provide a sufficient condition on the sequence $\lambda$ under which the Julia set $\mathcal{J}$ coincides with the entire complex plane. We first examine a property of the escaping set, generalizing the classical autonomous case to the non‑autonomous setting. Building on this property, we then give a complete proof of Theorem \ref{thm-whole-plane}.
\begin{theorem}\label{thm-escaping-Julia}
Let  $f$ be a transcendental entire function in the Eremenko-Lyubich class, and let $\lambda=\{\lambda_n\}_{n=1}^{\infty}$ be a sequence satisfying
 $0<\min\limits_{n}|\lambda_n|\leq \max\limits_{n}|\lambda_n|<\infty$. Consider the non-autonomous iteration system $\{f_{\lambda_n}\}_{n=1}^{\infty}$, where $f_{\lambda_n}=\lambda_nf$. Then, 
$$\mathcal{I}(\{f_{\lambda}^n\}_{n=1}^{\infty})\subset \mathcal{J}(\{f_{\lambda}^n\}_{n=1}^{\infty}),$$ where $$\I:=\{z\in\C: \lim_{n\to\infty}f_{\lambda}^n(z)=\infty\}$$ means the escaping set of the non-autonomous iterations.
\end{theorem}
\begin{proof}
Choose $R>0$ so that the disk
$D_{R}=\{z:|z|<R\}$ contains all singularities of the family $\{f_{\lambda_n}\}_{n=1}^{\infty}$.
Define $$W_n=f_{\lambda_n}^{-1}(\mathbb{C}\backslash \overline{D_R}),\qquad \mathcal{V}_n=\exp^{-1}(W_n)$$
and let \( F_{\lambda_n} \) be the logarithmic transform of \( f_{\lambda_n} \), i.e.,
\[
\exp(F_{\lambda_n}(z)) = f_{\lambda_n}(\exp(z)),\quad
\exp(F^n_{\lambda}(z)) = f^n_{\lambda}(\exp(z))\ \text{for all}\ n\in\N.
\]
Take any \( z_0 \in \mathcal{I}(\{f_{\lambda}^n\}_{n=1}^{\infty}) \). 
Suppose, for a contradiction, that $z_0\in\F$. Then, there exists a disk $D_0=D(z_0, r), r>0$ in which the iterates $\{f_{\lambda}^n\}_{n=1}^{\infty}$ tend uniformly to $\infty$.

Consequently, there exists \( k_0 > 0 \) such that for every $z\in D_0$,
\[
|f_{\lambda}^{k}(z)| \geq R^2 \quad \text{for all } k \geq k_0.
\]
In particular, for all \( n \in \mathbb{N} \),
\[
|f_{\lambda}^{n+k_0}(z)|=|f_{\sigma^{k_0}\lambda}^{n}(f_{\lambda}^{k_0}(z))| \geq R^2.
\]
Choose \( w \) with \( \exp(w) = z \). Then
\[
\Re \bigl( F_{\lambda}^{n+k_0}(w) \bigr) \geq 2\log R.
\]

A standard estimate (\cite[Lemma 1]{EL92}) gives, for all \( n \geq 0 \),
\[
|F_{\lambda_{k_0+n+1}}'(F_{\lambda}^{k_0+n}(w))| \geq \frac{\Re F_{\lambda}^{k_0+n+1}(w) - \log R}{4\pi} \geq \frac{\log R}{4\pi}.
\]
It follows that
\[
\lim_{n \to \infty} \bigl| (F_{\sigma^{k_0}\lambda}^n)'(F_{\lambda}^{k_0}(w)) \bigr| = \infty.
\]

By Koebe’s one-quarter theorem, this implies
\[
f_{\lambda}^{k_0}(z) \notin \mathcal{F}\bigl( \{f_{\sigma^{k_0}\lambda}^n\}_{n=1}^{\infty} \bigr).
\]
Indeed, if $f_{\lambda}^{k_0}(z)$ were in the Fatou set, there would exist $r_0>0$ such that
$$D(f_{\lambda}^{k_0}(z), r_0)\subset \mathcal{F}\bigl( \{f_{\sigma^{k_0}\lambda}^n\}_{n=1}^{\infty} \bigr).$$
Lifting through the exponential, we could then find a disk $V=D(F_{\lambda}^{k_0}(w), r_{1})$ in $\mathcal{V}_{k_0+1}$ (with $r_1>0$) such that for every $v\in V$,
$$\Re F_{\sigma^{k_0}\lambda}^n(v)\geq 2\log R\quad \text{and}\quad
\lim_{n\rightarrow\infty}
F_{\sigma^{k_0}\lambda}^n(v)=\infty. 
$$

Since $F_{\sigma^{k_0}\lambda}^n$ is conformal on $V$, Koebe’s one‑quarter theorem implies that $V' = F_{\sigma^{k_0}\lambda}^n(V)$ contains a disk of radius at least $\dfrac{1}{4}|(F_{\sigma^{k_0}\lambda}^n)'(F_{\lambda}^{k_0}(w))|
  r_1>2\pi$ for large $n$. However, as $V'\subset \mathcal{V}_{k_0+n+1}$ and each $\mathcal{V}_m$ lies in a strip of height at most $2\pi$, the holomorphic disk $V'$ cannot contain any disk of radius greater than $2\pi$. This is a contradiction.
Thus $$f_{\lambda}^{k_0}(z_0)\in\mathcal{J}\bigl(\{f_{\sigma^{k_0}\lambda}^n\}_{n=1}^{\infty}\bigr).$$ 
On the other hand, it is clear that $$(f_{\lambda}^k)^{-1}(\mathcal{J}(\{f_{\sigma^{k+j}\lambda}^n\}_{n=1}^{\infty}))= 
\mathcal{J}(\{f_{\sigma^j\lambda}^n\}_{n=1}^{\infty})$$ for all $k, j\in\N$, hence $z_0\in\mathcal{J}\bigl(\{f_{\lambda}^n\}_{n=1}^{\infty}\bigr)$.
We have therefore shown that every escaping point belongs to the Julia set, i.e., 
\[
\mathcal{I}\bigl( \{f_{\lambda}^n\}_{n=1}^{\infty} \bigr) \subset \mathcal{J}\bigl( \{f_{\lambda}^n\}_{n=1}^{\infty} \bigr),
\]
as required.

\end{proof}

Now, we are ready to prove Theorem \ref{thm-whole-plane}.
\begin{proof}[Proof of Theorem \ref{thm-whole-plane}]
We first show that $\mathbb{R}\subset \I$.  For any $x\in [0, \infty)$, we have
$$f_{\lambda_n}(x)=\lambda_n \exp(x)\geq 
\exp(x-1)\geq x.$$
Hence, $f_{\lambda_n}(x)\geq x$ and
  $\{f_{\lambda}^n(x)\}_{n=1}^{\infty}$ is an increasing sequence. 
Let $b=\lim\limits_{n\rightarrow\infty}f_{\lambda}^n(x)$; then $b\geq 1$.  We now show that $b=\infty$. 

Assume, to the contrary, that $b<\infty$. 
From $$
\lambda_{n-1}=\frac{f_{\lambda}^n(x)}{\exp(f_{\lambda}^{n-1}(x))},
$$
we obtain
$$
\lim_{n\rightarrow\infty}\lambda_{n}=\frac{b}{\exp b}\leq \frac{1}{e},
$$
which contradicts the hypothesis
$$
\limsup\limits_{n\rightarrow\infty}\lambda_{n}>\frac{1}{e}. 
$$
Hence,
$$
\lim_{n\rightarrow\infty}f_{\lambda}^n(x)=b=\infty,
$$
so
$\mathbb{R}\subseteq \I$.
Indeed, by the periodicity of the exponential map $\exp(z)$, we have 
$$
\mathbb{R}+2k\pi i\subset \I
$$
for any $k\in\mathbb{Z}$. By Theorem \ref{thm-escaping-Julia},
we therefore have
\begin{equation}\label{eqn-R-Julia}
\mathbb{R}+2k\pi i\subset \J,\qquad k\in\mathbb{Z}
\end{equation}
Set
$$
\mathcal{Z}=\bigcup_{k\in \mathbb{Z}}\mathbb{R}+2k\pi i,\ Y=\bigcup_{k\in \mathbb{Z}}\mathbb{R}+k\pi i.
$$ 
Now take $z_0\in \mathcal{F}(\f)$, We shall prove that
 \begin{equation}\label{eq-eventual-limit}
 \lim_{n\rightarrow\infty}\Im z_{n}=0,
 \end{equation}
 where $z_n=\fl^n(z_0)$, and $\Im z_{n}\neq 0$ for all $n\in \mathbb{N}$.
In fact we establish a stronger statement: 
\begin{equation}\label{eq-product}
1\leq \prod_{n=1}^{\infty}\frac{|\Im z_{n}|}{|\sin \Im z_{n}|}<\infty
\end{equation}
which directly implies \eqref{eq-eventual-limit}. 

Let $W$ be a Jordan disk centered at $z_0$ with $W\subset \mathcal{F}(\f)$.  
Define $W_0=W$ and $W_{n}=\fl^n(W)$ for $n\geq 1$, and let $U=\{z: \Im z\neq 0\}$. 
From (\ref{eqn-R-Julia}), we see that each $W_n$ lies in a connected component of $\mathbb{C}\backslash Y$. 
For each $n$, the map $f_{\lambda_n}$ is conformal from $W_{n-1}$ onto $W_{n}$; moreover, it conformally sends each strip 
$$
U_{k}=\{z\in \mathbb{C}: k\pi<\Im z<(k+1)\pi\}
$$ to either the upper or the lower half‑plane, depending on the parity of $k$. 

For a conformal map $f:U\to V$, define its gradient by
$$
D_U^Vf(z)=\lim\limits_{w\rightarrow z}\frac{d_{V}(f(z), f(w)}{d_{U}(z, w)},
$$ where $d_U, d_V$ denote the hyperbolic distances.
One readily checks that
$$
D_{U}^V f(z)=\frac{\rho_{V}(f(z))|f'(z)|}{\rho_{U}(z)}. 
$$
By the Schwarz-Pick lemma, we have
$$
D_{W}^{W_n}f_{\lambda}^{n}(z_0)=1.
$$
Applying the chain rule gives
$$
D_{W}^{W_n}f_{\lambda}^{n}(z_0)=
D_{W}^{U}Id(z_0)\left(\prod_{j=1}^{n}
D_{U}^{U}f_{\lambda_j}(z_{j-1})\right)
D_{U}^{W_n}Id(z_n).
$$
Again by the Schwarz-Pick lemma, we know that 
$$D_{U}^{W_n}Id(z_n)\geq 1.$$
A direct computation yields 
$$
D_{U}^{U}f_{\lambda_j}(z_{j-1})=
\frac{|(f_{\lambda_j})'(z_{j-1})\Im z_{j-1}|}{|\Im f_{\lambda_j}(z_{j-1})|}=
\frac{|\Im z_{j-1}|}{|\sin \Im (z_{j-1})|}.
$$
Combining these estimates together yields inequality \eqref{eq-product}.

Next, we will show that $$\lim_{n\rightarrow\infty}\Re z_n=\infty. $$ 
Recall that 
$$|z_{n+1}|=\lambda_n\exp(\Re z_n)>\frac{1}{e}\exp(\Re z_n),\qquad \Re z_{n+1}=\lambda_n\exp(\Re z_n)\cos(\Im z_n).$$ 
Hence,
$$
\liminf\limits_{n\rightarrow\infty}\Re z_n\geq \frac{1}{e}.
$$

We claim that
$
\liminf\limits_{n\rightarrow\infty}\Re z_n
$ is either $1$ or $\infty$. 

Assume that $\liminf\limits_{n\rightarrow\infty}\Re z_n=b$, where $\frac{1}{e}\leq b<\infty$ and $b\neq 1$. Then, there exists a subsequence $z_{n_k}$ such that
$$\lim_{k\rightarrow\infty}\Re z_{n_k}=\lim_{k\to\infty}\lambda_{n_k-1}\exp \Re z_{n_k-1}\cos(\Im z_{n_k-1})=b.$$ 
This implies that
$$
\lim_{k\rightarrow\infty}\lambda_{n_k-1}\exp \Re z_{n_k-1}=b.
$$
Consequently
$$\liminf_{k\rightarrow\infty}\Re z_{n_{k}-1}\leq \log b+1<b,
$$
contradicting the definition of $b=\liminf\limits_{n\rightarrow\infty}\Re z_n$.  

A similar argument shows that 
$
\limsup\limits_{n\rightarrow\infty}\Re z_n
$ is either $1$ or $\infty$.

Moreover,
$
\limsup\limits_{n\rightarrow\infty}\Re z_n>1
$ 
if and only if
$
\liminf\limits_{n\rightarrow\infty}\Re z_n>1$, which follows from the basic inequality
$$
\frac{\exp(1+b)}{e(1+b)}>1
$$
for any $b\geq 0$.
Thus,
$\lim\limits_{n\rightarrow\infty}\Re z_n$ is either $1$ or $\infty$. However, if 
$\lim\limits_{n\rightarrow\infty}\Re z_n=1$, then from $$\Re z_{n+1}=\lambda_n\exp(\Re z_n)\cos(\Im z_n),$$ we would obtain that $\lim\limits_{n\rightarrow\infty}
\lambda_n=\frac{1}{e}$, contradicting our assumption $\lim\limits_{n\rightarrow\infty}
\lambda_n>\frac{1}{e}$. Therefore,
$$
\lim\limits_{n\rightarrow\infty}\Re z_n=\infty.
$$ 
It follows from (\ref{eq-eventual-limit}) and Theorem \ref{thm-escaping-Julia} that
$$z_0\in\I\subset\J,$$ which is impossible, because $z_0\in\F$.

\end{proof}

\section{Covering properties of the exponential families}\label{sec:covering}

In this section, based on the Ahlfors island theorem, we establish several covering properties for exponential maps, including a blow-up type of non-autonomous Julia sets.

\begin{lemma}[Uniform Expansion Property]\label{lem-expand}
Let $\lambda$ be a sequence in $[A, B]^{\N},\ A>\frac{1}{e}$. Let $U$ be a domain and $K\subset \mathbb{C}\backslash\{0\}$ compact.
Then, there exists $m\in \mathbb{N}$ depending on $U, K$ and $\lambda$ such that
$$
f_{\lambda}^n(U)\supseteq K
$$ 
for all $n\geq m$.
\end{lemma}
\begin{proof}
We first show that there exists $R_0>0$ such that for every $\mu\in [A, B]$ and every $R>R_0$, 
$$
f_{\mu}\left(A\left(\frac{1}{R}, R\right)\right)\supseteq A\left(\frac{1}{R}, R\right)$$ where $A\left(\frac{1}{R}, R\right):=\{z: \frac{1}{R}\leq |z|\le R\}$.

Define two rectangles
$$W_{\mu}:=\{z: -\log R-\log\mu\le\Re z\leq \log R-\log\mu,\ 2\pi\leq \Im z\leq 4\pi\},
$$
$$
V_{R}:=
\{z: -\log R-\log B\leq \Re z\leq \log R-\log A, 2\pi\leq \Im z\leq 4\pi\}.
$$
Then we have
$$
W_{\mu}\subset V_R\quad\text{and}\quad f_{\mu}(W_{\mu})=A\left(\frac{1}{R}, R\right).
$$
For any $z\in V_{R}$,
$$
2\pi\leq |\Im z|\leq |z|\leq |\Im z|+|\Re z|
\leq \log R+|\log A|+|\log B|+4\pi.
$$
Hence, we can choose $R_0>0$ such
 that $$V_{R}\subset A\left(\frac{1}{R}, R\right),\ \text{for all}\ R>R_0.$$  
Consequently, for every $R>R_0$ and every $\mu\in [A, B]$, 
 $$f_{\mu}\left(A\left(\frac{1}{R}, R\right)\right)\supset f_{\mu}(V_{R})\supset
 f_{\mu}(W_{\mu})=A\left(\frac{1}{R}, R\right).$$ 
Now, fix a small disk  $D(z_0, \epsilon)\subset U$ and take $R$ large enough such that $K\subseteq A(\frac{1}{R}, R)$.
Since $\{f_{\lambda}^n\}_{n=1}^{\infty}$ is not normal in $D(z_0, \epsilon)$, the Ahlfors' three islands theorem (\cite{Ahlfors35, Berg00}) yields an integer $n_1$ and a subdomain $U_{n_1}$ of $D(z_0, \epsilon)$ such that 
$f_{\lambda}^{n_1}$ maps $U_{n_1}$ conformally onto one of three disjoint Jordan domains $V_R, V_R+4\pi i, V_R+8\pi i$. 
It follows that $$f_{\lambda}^{n_1+1}(U_{n_1})\supset A\left(\frac{1}{R}, R\right)$$
Hence, for every $k\in \mathbb{N}$, $$f_{\lambda}^{n_1+k}(U_{n_1})\supset A\left(\frac{1}{R}, R\right)\supseteq K.$$
\end{proof}

Using a compactness argument we obtain the following statement.
\begin{lemma}\label{lemma-uniform}
Let $D$ be a bounded Jordan domain and $0\notin \overline{D}$ and let $\frac{1}{e}<A<B$.
For any $\epsilon>0$ and  $R>0$, there exists $m=m(A, B, \epsilon, R)$ such that for every $z\in \overline{D(0, R)}$, every sequence $\lambda=(\lambda_n)_{n=1}^{\infty}\in [A, B]^{\mathbb{N}}$ and every $n\geq m$,
\begin{equation}\label{covering}
f_{\lambda}^{n}(D(z, \epsilon))\supset D.
\end{equation}
Moreover, define for $n\ge m$
$$
V=V_{z, \epsilon, \lambda, n}=\{w\in D(z, \epsilon): f_{\lambda}^{n}(w)\in D\}.
$$
Then, there exists $c=c(R, A, B, D, n, \epsilon)>0$ such that
 $$\mathrm{Area}(V)\ge c.$$
\end{lemma}
\begin{proof}
Now, (\ref{covering}) follows from Lemma \ref{lem-expand} directly. To prove the area estimate, note that by the area formula,
$$
\mathrm{Area}(D)\leq \int_{V}|(f_{\lambda}^{n})'(w)|^2\frac{i dw\wedge d\overline{w}}{2}.
$$
For a fixed $n\in \mathbb{N}$, the derivative
$(f_{\lambda}^{n})'$ is jointly continuous with respect to $[A, B]^{\mathbb{N}}\times \mathbb{C}$, it follows from the compactness that  $$M=\sup\{|(f_{\lambda}^n)'(w)|: w\in D(0, R+2\epsilon), \lambda_i\in [A, B], 1\leq i\leq n\}<\infty.$$
Thus,
$$
\text{Area}(D)\leq M^2 \text{Area}(V).
$$
So,
$$
\text{Area}(V)\geq \frac{\text{Area}(D)}{M^2}.
$$
\end{proof}

Given two pairs of Jordan domains $(U, U')$ and $(V, V')$, we say a map $g$ is conformal from $(U, U')$ to $(V, V')$ if $g$ maps $U$ conformally onto $V$ and $U'$ conformally onto $V'$. 
\begin{lemma}[Backward contracting lemma]\label{lem-bcl}
Let $\lambda\in [A, B]^{\N}$,
with $\frac{1}{e}<A<B<\infty$.
Let $W$ and $V$ be bounded domains in $\mathbb{C}$ with $\overline{W}\subset V$. 
Given a sequence $\{f_{\lambda}^{n_k}\}_{k=1}^{\infty}$, if there exist two sequence of Jordan domains $\{U_{k}\}_{k=1}^{\infty}$ and $\{U_{k}^{'}\}_{k=1}^{\infty}$ contained in a compact set $K\subset\mathbb{C}$ such that for each $k$, $\fl^{n_k}$ is conformal from the pair $(U_{k}, U_{k}^{'})$ to the pair $(W, V)$. Then
$$
\lim_{k\rightarrow\infty}\mathrm{diam}(U_{k})=0,
$$ here, as usual, for any set $U$, $\mathrm{diam}(U):=\sup\{|x-y|: x, y\in U\}$.
\end{lemma}
\begin{proof}
For each $k\in \mathbb{N}$, let $g_k:V \to U_k'$ the branch of the inverse of $f_{\lambda}^{n_k}$. By Montel's theorem,  $\{g_k\}_{k=1}^{\infty}$ is a normal family on $V$.

Suppose the contrary, after passing to a subsequence, we may assume that $g_{k}$ converges locally uniformly to a non-constant analytic function $g_{\infty}$ on $V$.

Pick $a\in g_{\infty}(W)$, and choose $r>0$ such that $\overline{D(a, r)}\subset g_{\infty}(W)$. By Rouche's theorem, for $k$ sufficiently large, we have $D(a, r)\subset g_{k}(W)\subseteq U_k$. Consequently, 
$\fl^{n_k}(D(a, r))\subset W$ for large $k$. However,  Lemma \ref{lemma-uniform} together with Theorem \ref{thm-whole-plane} implies that $D(a, r)\subset \mathcal{J}(\f)$, so the iterates $\fl^{n_k}$ cannot eventually map $D(a, r)$ into the bounded domain $W$. This contradiction proves $$
\lim_{k\rightarrow\infty}\mathrm{diam}(U_{k})=0.
$$
\end{proof}

The basic technique of the main result in this paper relies on the following classical Koebe distortion theorem.
\begin{theorem}[Koebe's distortion theorem \cite{Gol69}]
Let $g$ be a univalent holomorphic function in the disk $D(z_0, r)$ and $0<\alpha<1$. Then
\begin{enumerate}
    \item[(i)] $|g'(z_0)|\dfrac{\alpha r}{(1+\alpha)^2}\leq|g(z)-g(z_0)|\leq |g'(z_0)|\dfrac{\alpha r}{(1-\alpha)^2}, z\in D(z_0, \alpha r)$;
    \item[(ii)] $\dfrac{(1-\alpha)^4}{(1+\alpha)^4}\le\left|\dfrac{g'(z_1)}{g'(z_2)}\right|\leq \dfrac{(1+\alpha)^4}{(1-\alpha)^4}, z_1,z_2\in D(z_0,\alpha r)$.
\end{enumerate}
\end{theorem}

\section{A typical trajectory of the non-autonomous iteration of exponential maps}\label{sec:orbit}
 
For the remainder of this paper, unless otherwise stated, we always assume that the sequence $\lambda$ lies in $[A, B]^{\mathbb{N}}$ with $1/e<A<B<\infty$. Consequently, Theorem \ref{thm-whole-plane} implies that the Julia set coincides with the whole plane, i.e., $\J=\C.$

Our first result shows that points whose $\omega$‑limit sets are the entire plane are abundant in the sense of Baire category. This observation could, in fact, be regarded as a precursor of
Theorem \ref{thm-main} from the category points of view.
\begin{theorem}
$\{z\in \mathbb{C}:\omega_{\lambda}(z)=\mathbb{C}\}$
is of second category. 
\end{theorem}
\begin{proof}
For each $n$, choose a covering of $\C$ by a sequence of open sets $\{U_{n, k}\}_{k=1}^{\infty}$ with $\mathrm{diam} (U_{n, k})\leq \frac{1}{n}$.
Define
$$S=\mathbb{C}\backslash\left(\bigcup_{k=1}^{\infty}\bigcup_{n=1}^{\infty}
\bigcup_{j=1}^{\infty}\bigcap_{m=j}^{\infty}
(\fl^{m})^{-1}(\mathbb{C}\backslash U_{n, k})\right).$$
By Baire’s category theorem, it suffices to prove that each set
$\displaystyle\bigcap_{m=j}^{\infty}(\fl^{m})^{-1}(\mathbb{C}\backslash U_{n, k})$ is closed and nowhere dense. The closedness is obvious. 
Moreover, from Lemma \ref{lemma-uniform}, we know that such an intersection contains no interior point; hence, it is nowhere dense.
\end{proof}

The main result of this paper, Theorem \ref{thm-main}, will be deduced from the following key estimate via a pull‑back argument.
\begin{theorem}\label{thm-key}
There exist a bounded domain $D$ and a constant $C=C(A, B; D)>0$ such that for any sequence $\lambda\in [A, B]^{\mathbb{N}}$, 
$$
\mathrm{Area}(D_{\lambda})\geq C,
$$
where 
$$
D_{\lambda}:=\{z\in D: \omega_{\lambda}(z)\subset \mathcal{B}_{\lambda}\}.
$$
\end{theorem}

We now prove Theorem~\ref{thm-main} assuming Theorem~\ref{thm-key}.
\begin{proof}[Proof of Theorem \ref{thm-main}]
By Theorem \ref{thm-key}, there exists an absolute constant $C>0$ such that for every $\lambda=(\lambda_n)_{n=1}^{\infty}$,
$$\mathrm{Area}(D_{\sigma^{k}\lambda})\geq C, \quad\text{for all}\ k\geq 0.
$$
Define the union of the non-typical points
\begin{equation*}\label{eq-W}
W_{\lambda}:=\{z\in \mathbb{C}: \omega_{\lambda}(z)\not\subset \mathcal{B}_{\lambda}\}.
\end{equation*}
Recall that $\omega_{\sigma^k\lambda}(f_{\lambda}^k(z))=
\omega_{\lambda}(z)$ and $\mathcal{B}_{\sigma^k \lambda}\subset\mathcal{B}_{\lambda}$. 
Consequently,
\begin{equation}\label{eqnfn}
(f_{\lambda}^{n})^{-1}(\mathbb{C}\backslash W_{\sigma^n
\lambda})\subset \mathbb{C}\backslash W_{\lambda},\quad\text{for all}\ n\in\N.
\end{equation}
Suppose, for contradiction, that $\mathrm{Area}(W_{\lambda})>0$. Choose a Lebesgue density point $z_0\in W_{\lambda}$, that is, 
$$
\lim_{r\rightarrow 0}\frac{\mathrm{Area}(W_{\lambda}\cap B(z_0, r))}{\mathrm{Area}( B(z_0, r))}=1.
$$
Since $z_0\in W_{\lambda}$,  there exists a subsequence $\{n_k\}_{k=1}^{\infty}$ tending to infinity such that 
\begin{equation}\label{eq-far}
\inf_{k\geq1}\inf_{w\in \mathcal{B}_{\lambda}}|f_{\lambda}^{n_k}(z_0)-w|>2\delta
\end{equation}
and
$$
\sup_{k\geq 1}|f_{\lambda}^{n_k}(z_0)|\leq M
$$
for some $\delta>0$ and $M>0$. By \eqref{eq-far},
$$
\text{Sing}((f_{\lambda}^{n_k})^{-1})\cap D(f_{\lambda}^{n_k}(z_0), 2\delta)=\emptyset.
$$
 Up to a subsequence, we may assume that 
$$\lim_{k\rightarrow\infty}f_{\lambda}^{n_k}(z_0)=x_0\quad \text{for some}\ x_0\in\C.$$ 
Thus, for sufficiently large $k$, we have 
$$f_{\lambda}^{n_k}(z_0)\in D\left(x_0, \frac{1}{2}\delta\right)\subset D(f_{\lambda}^{n_k}(z_0), \delta) \subset D\left(x_0, \frac{3}{2}\delta\right).$$

Now, let $U_{k}(z_0)$ be the connected component of $(f_{\lambda}^{n_k})^{-1}(D(f_{\lambda}^{n_k}(z_0), \delta))$ that contains $z_0$. Similarly, let $V_{k}(z_0)$ and $V_k'(z_0)$ be the components of 
$$ (f_{\lambda}^{n_k})^{-1}\left(D\left(x_0, \frac{3}{2}\delta\right)\right)\quad\text{and}\quad (f_{\lambda}^{n_k})^{-1}\left(D\left(x_0, \frac{7}{4}\delta\right)\right)$$ 
respectively, containing $z_0$. 
By Koebe's distortion theorem, each
$U_{k}(z_0)$ is a Jordan domain containing a round disk of radiu at least $\frac{1}{C}\text{diam}(U_k(z_0))$, for an absolute constant $C>0$..

Applying Lemma \ref{lem-bcl} to the conformal maps $f_{\lambda}^{n_k}$ on the pair $(V_k(z_0), V'_k(z_0))$, we obtain
$$\lim_{k\rightarrow\infty}\text{diam}(V_{k}(z_0))=0, \ \text{and hence} \
\lim_{k\rightarrow\infty}\text{diam}(U_{k}(z_0))=0.
$$
As $z_0$ is a density point of $W_{\lambda}$, it follows that
\begin{equation}\label{eqn-density}
\lim_{k\rightarrow\infty}
\frac{\mathrm{Area}(W_{\lambda}\cap U_{k}(z_0))}{\mathrm{Area}(U_{k}(z_0))}=1.
\end{equation}
Set $K:=\overline{D(0, M+3\delta)}$. By Lemma \ref{lem-expand}, there exists an integer $N\in\N$ such that for any $z\in K$,
$$D\subset f_{\sigma^{n_k}\lambda}^{N}(D(z, \frac{1}{2}\delta)),$$
where $D$ is the domain from Theorem \ref{thm-key}.
In particular, 
$$D\subset f_{\sigma^{n_k}\lambda}^{N}\left(D(x_0, \frac{1}{2}\delta)\right).$$
Define
$$
V:=\left\{w\in D(x_0, \frac{1}{2}\delta), f_{\sigma^{n_k}\lambda}^{N}(w)\in D\backslash W_{\sigma^{n_k+N}\lambda}\right\}.
$$
Lemma \ref{lemma-uniform} provides an absolute constant $c>0$ such that
$$
\frac{\mathrm{Area}(V)}
{\mathrm{Area}(D(x_0, \frac{\delta}{2}))}>c.
$$
Consequently, there exists another absolute constant $c'>0$ with
$$
\frac{\mathrm{Area}(V)}
{\mathrm{Area}(D(f_{\lambda}^{n_k}(z_0), \delta))}>c',
$$
and then
$$
\frac{\mathrm{Area}(D(f_{\lambda}^{n_k}(z_0), \delta)\cap 
(f_{\sigma^{n_k}\lambda}^{N})^{-1}(D\backslash W_{\sigma^{n_k+N}\lambda}))}{\mathrm{Area}
(D(f_{\omega}^{n_k}(z_0), \delta))}>c'.
$$
Applying Koebe’s distortion theorem once more, we find a constant $c''>0$ such that
$$
\frac{\mathrm{Area}(U_{k}(z_0)\cap (f_{\lambda}^{n_k+N})^{-1}(D
\backslash W_{\sigma^{n_k+N}\lambda}))}
{\mathrm{Area}(U_{k}(z_0))}\geq c^{''}.
$$
Hence, from (\ref{eqnfn}), it follows that 
$$
\frac{\mathrm{Area}(U_{k}(z_0)\cap (\C\backslash W_{\lambda}))}
{\mathrm{Area}(U_{k}(z_0))}\geq c^{''},
$$
But this yields that
$$
\limsup_{k\rightarrow\infty}\frac{\mathrm{Area}(U_{k}(z_0)
\cap W_{\lambda})}
{\mathrm{Area}(U_{k}(z_0))}\leq 1-c^{''}.
$$
contradicting (\ref{eqn-density}). Therefore, $\mathrm{Area}(W_{\lambda})=0$, completing the proof.
\end{proof}

\section{Geometric construction of non-autonomous system of exponential maps}\label{sec:loss}
Let $K\subset\C$ be a compact set. We say that $K$ is of \emph{$C$-bounded shape} for a constant $C>1$ if $K$ contains a round disk of radius at least $\frac{1}{C}\text{diam} (K)$.

For a Jordan domain $D$ of $C$-bounded shape, elementary geometry gives the area bounds
$$
\frac{\pi}{C^2}(\text{diam}(D))^{2}\leq \text{Area}(D)\leq \pi(\text{diam}(D))^{2}.
$$

Fix a sufficiently large constant $s>0$, and let $\mathcal{H}_s:=\{z:\Re z>s\}$. Let $D\subset \mathcal{H}_s$ be a Jordan disk of $C$-bounded shape. For a sequence $\mu=(\mu_{n})_{n=1}^{\infty}\in [A, B]^{\mathbb{N}}$,
we say that a point $z\in D$ is of \emph{type 0} (with respect to $\mu$ and $s$)
if one of the following conditions holds:
\begin{enumerate}
\item $|\Re f_{\mu_1}(z)|\leq \exp(c_1 s)$;
\item $d(f_{\mu_1}(z), \mathcal{B}_{\sigma^k\mu})\leq 1$ for some $k\in\mathbb{N}$,
\end{enumerate} 
where $c_1\in(0,1)$ is an absolute constant. By $\mathcal{T}_0(D)$ we mean the set of all type $0$ points in $D$.

\subsection{Size Estimate for Type 0 Points}

We now give a quantitative estimate for the area occupied by type 0 points.
\begin{proposition}\label{pro-type0}
Let $D\subset\mathcal{H}_s$ be a bounded Jordan domain of $L$-bounded shape with $\mathrm{diam}(D)\geq \ell>0$.  Then the area of $\mathcal{T}_0(D)$ satisfies
\begin{equation}\label{eqnt0}
\mathrm{Area}(\mathcal{T}_0(D))\leq C(\exp(-s)+\exp((c_1-1)s))\mathrm{Area}(D).
\end{equation}
where the constant 
$C$ depends only on $A, B, L, \ell$.
\end{proposition}
\begin{proof}
Assume that a point $z\in D$ satisfies (ii). Then, there exist $\tau_1, \dots, \tau_k\in[A, B]$ such that $$|f_{\mu_1}(z)-f_{\tau_k}\circ \cdots f_{\tau_1}(0)|\leq 1.$$ 
Thus, 
$$f_{\tau_k}\circ\cdots f_{\tau_1}(0)\geq |f_{\mu_1}(z)|-1
\geq \mu_1\exp(s)-1\geq \mu_1\exp(s-1)$$
provided that $s$ is large enough (e.g. $s\geq \ln\frac{e}{A(e-1)}$). Hence, every pre-image of  $f_{\tau_k}\circ\cdots f_{\tau_1}(0)$ under $f_{\mu_1}$ lies in a union of horizontal lines:
$$f_{\mu_1}^{-1}(f_{\tau_k}\circ \cdots f_{\tau_1}(0)) \in \bigcup_{k=1}^{\infty}\{z: \Re z>s-1, \Im z=2k\pi i\}.$$ 
Let $\Gamma=\Gamma(t) (0\leq t\leq 1)$ be the straight line segment joining $f_{\mu_1}(z)$ to $f_{\tau_k}\circ \cdots f_{\tau_1}(0)$.  
For $s>\ln\frac{e}{A}$, the segment $\Gamma$ avoids the origin, hence its pre-image $f_{\mu_1}^{-1}(\Gamma)$ consists of countably many connected components. 
Write $f_{\mu_1}^{-1}(\Gamma)=\{\gamma_{i}\}_{i=1}^{\infty}$, where each $\gamma_i$ is mapped bijectively onto $\Gamma$. There exists an index $i_0$ such that $z$ is an endpoint of $\gamma_{i_0}$; denote the other endpoint by $b=b_{i_0, \tau_1\dots\tau_k}$, so that $f_{\mu_1}(b)=f_{\tau_k}\circ \cdots\circ f_{\tau_1}(0)$.
Therefore, 
$$
|z-b|\leq \int_{\gamma_{i_0}}|dw|\leq \frac{|\Gamma|}{Ae^{2-s}}
\leq \frac{1}{A}\exp(2-s).
$$
Define disks $D_{\tau_1\cdots \tau_k}:=D(b_{i_0, \tau_1\dots\tau_k}, \frac{1}{A}\exp(2-s))$ and set
$$
\mathcal{T}_{1}:=\bigcup_{n=1}^{\infty}\bigcup_{k=1}^{\infty}
\bigcup_{\tau_1\cdots\tau_k=[\sigma^n\lambda]_{k}}D_{\tau_1\cdots \tau_k},
$$
where $[\sigma^n\lambda]_{k}$ means the first $k$ terms of the sequence $\sigma^n\lambda$.

For any $x\in\R$, let $\mathcal{T}_{1}^{x}:=\{z\in \mathcal{T}_1: \Re z=x \}$. Since each component $U_x$ of $\mathcal{T}_{1}^{x}$ lies within distance $\frac{1}{A}e^{2-s}$ of some line $\{\Im z=2k\pi\}$,
it follows that $\mathcal{T}_1^{x}$ has at most $\frac{\text{diam} (D)}{2\pi}+1$ connected components and $\mathrm{Length}(U_{x})\leq\frac{2}{A}e^{2-s}$.
Applying Fubini’s theorem, we obtain
\begin{align*}
\text{Area}(\mathcal{T}_{1})&=\int_{x\in \mathbb{R}}\mathrm{Length}(\mathcal{T}_1^{x})dx\\
&\leq \frac{
2e^{2-s}}{A}\left(\frac{\text{diam} (D)}{2\pi}+1\right)\text{diam}(D)\\
(\text{by}\ \text{diam}(D)\ge \ell)\ &\leq \frac{2 e^{2-s}}{A}\left(\frac{1}{2\pi}+\frac{1}{\ell}\right)(\text{diam}(D))^2\\
(\text{as}\ D\ \text{is of L-bounded shape})\ &\leq \frac{2 e^2L^2}{\pi}\left(\frac{1}{2\pi}+\frac{1}{\ell}\right)\text{Area}(D)e^{-s}.
\end{align*}
 Thus
$$\text{Area}(\mathcal{T}_{1})\leq C_1 e^{-s}\text{Area}(D),\quad \text{where}\quad C_1:=\frac{2 e^2L^2}{\pi}\left(\frac{1}{2\pi}+\frac{1}{\ell}\right).$$

Now, define $\mathcal{T}_2=\{z\in D:  \Re f_{\mu_1}(z)\leq \exp(c_1s)\}.$
Then, 
$$
\mathcal{T}_2\subset \bigcup_{n=1}^{\infty}
\left\{z\in D: |\Im z-n\pi+\frac{\pi}{2}|\leq \frac{\pi}{2\mu_1}\exp((c_1-1)s)\right\}.
$$

Arguing as for $\mathcal{T}_1$ (again via Fubini's theorem), we obtain a constant $C_2>0$ such that
$$
\text{Area}(\mathcal{T}_2)\leq C_2\exp((c_1-1)s)\text{Area}(D).
$$
Together with the bound on $\operatorname{Area}(\mathcal{T}_1)$ and the partition $\mathcal{T}_0(D) = \mathcal{T}_1 \cup \mathcal{T}_2$, this establishes the desired estimate and concludes the proof.
\end{proof}

\subsection{Area Estimate for \texorpdfstring{$\epsilon$}{epsilon}-grid loss points}
This subsection is devoted to a quantitative estimate for the area of the so‑called $\epsilon$-grid loss points.

 We first note that the sets $\mathcal{T}_1$ and $\mathcal{T}_2$ introduced in the proof of Proposition \ref{pro-type0} can be covered by a controlled number of thin rectangles. More precisely, they are contained in the union of 
$$N:=\left[\frac{3\text{diam}(D)}{2\pi}\right]+2$$ disjoint rectangles $T_j$, each having length $\text{diam}(D)$ and width either $\frac{\pi}{2A}e^{(c_1-1)s}+\frac{2e^2}{A}e^{-s}$ or $\frac{\pi}{A}e^{(c_1-1)s}+\frac{2e^2}{A}e^{-s}$, where $c_1$ is the constant appearing in Proposition \ref{pro-type0}. Let $T=\cup_{j=1}^NT_j$.
Observe that $\mathcal{T}_0(D)\subset T,$ and the proof of Proposition \ref{pro-type0} actually  gives the following estimate
\begin{equation*}
\mathrm{Area}(T)\leq C(\exp(-s)+\exp((c_1-1)s))\mathrm{Area}(D).
\end{equation*}
for some constant $C>0$, which directly implies inequality  (\ref{eqnt0}).

Now, Cover the complex plane by the disjoint squares $\{R_{m,n}\}_{m, n\in \mathbb{Z}}$, $$R_{m, n}:=[m-1, m+1]\times [n-1, n+1]\quad (m,n\in\mathbb{Z}).$$ For a parameter $\mu \in [A, B]$ and and a Jordan domain $D\subset\mathcal{H}_s$, we now
define the enlarged index set $$\mathcal{I}_{D}:=\{(m, n)\in \mathbb{Z}^2: R_{m, n}\cap \partial f_{\mu}(D)\neq\emptyset, \ \text{or} \ R_{m, n}\cap \partial f_{\mu}(T)\neq\emptyset\}$$ 
and the corresponding $\epsilon${-}grid loss set
$$\mathcal{T}_{\mu, s}^{\text{grid}}=\{z\in D:f_{\mu}(z)\in R_{m, n}\ 
\text{for} \ (m, n)\in \mathcal{I}_{D}\}.$$
Using Koebe's distortion theorem, we obtain the following uniform estimate.
\begin{proposition}\label{pro-grid}
Given constants $l, L>0$ and $\delta\in (0, 1)$, there exists $s_0=s_0(A, B, l, L, \delta)$ such that for any pair of Jordan domains $(D, D')$ satisfying
\begin{enumerate}[(a)]
\item $D\subsetneq D'$, $\mathrm{diam}(D) \geq l$;
\item $\delta\leq \mathrm{mod}(D'\backslash \overline{D})\leq \dfrac{1}{\delta}$;
\item $\delta\leq \dfrac{\mathrm{Length}^2(\partial D)}{\mathrm{Area}(D)}\leq \dfrac{1}{\delta}$;
\item $D$ is $L$-bounded shape. 
\end{enumerate}
and for any $s>s_0$ with $D, D'\subset \mathcal{H}_s$, we have
$$
\frac{\mathrm{Area}(\mathcal{T}_{\mu, s}^{grid})}{\mathrm{Area}(D)}\leq C_3(\exp(-s)+\exp((c_1-1)s))
$$
for some constant $C_3>0$.
\end{proposition}

\begin{proof}
Since $D$ is of $L$-bounded shape, it contains a disk $D(z_0, r_0)$ with
$r_0\geq \frac{1}{L}\mathrm{diam}(D)$. 

Let $\phi:\Delta\to D'$ be a conformal map with 
$\phi(0)=z_0$ and $\phi'(0)>0$. By Schwarz's lemma, we know that $\phi'(0)\geq  \dfrac{1}{L}\mathrm{diam}(D)$. Set $W=\phi^{-1}(D)$.
We claim that $$\phi'(0)\leq 4 \exp(4\pi/\delta)\mathrm{diam}(D).$$
Otherwise, taking
$\hat{W}:=D(0, \exp(-4\pi/\delta))$, Koebe’s one‑quarter theorem gives
$$
\phi(\hat{W})\supset D\left(z_0, \frac{\phi'(0)\exp(-4\pi/\delta)}{4}\right)\supset D(z_0, \mathrm{diam}(D))\supset D.
$$
which would imply $W\subset \hat{W}$. Then, from the condition (b), we have
$$
\frac{2}{\delta}=\mathrm{mod}(\Delta\backslash \hat{W})\leq \mathrm{mod}(\Delta\backslash W)=\mathrm{mod}(D'\backslash\overline{D})\leq\delta^{-1},
$$
which is a contradiction. Hence the claim holds. 

Applying  Koebe's one-quarter theorem 
to $\phi^{-1}: D(z_0, r_0) \rightarrow \Delta$ yields
$$W\supset D\left(0, \rho_1\right),$$ 
where $\rho_1:=\dfrac{\exp(-\frac{4\pi}{\delta})}{16L}$.
On the other hand, by Grötzsch's module theorem, there exists $\rho_2=\rho_2(\delta)\in (0, 1)$ such that
$W\subset D(0, \rho_2).$ 

Now, let $V=D\backslash T$.
By construction of $T$, the map $f_\mu$ is injective on each connected component of $V$. It is easy to see that
$|f_{\mu}'(z)|\geq \exp(s-1)$ for all $z\in D$, consequently, $\mathcal{T}_{\mu,s}^{\mathrm{grid}}$ lies either in the $\sqrt{8}e^{1-s}$-neighbourhood of $\partial D$ or in the $\sqrt{8}e^{1-s}$-neighbourhood of $\partial T$.
Denote by $S_1$ the $\sqrt{8}e^{1-s}$-neighborhood of $\partial D$ in $D$, and by $S_2$ the $\sqrt{8}e^{1-s}$-neighborhood of $\partial T$ in $D$.
We shall estimate $\operatorname{Area}(S_1)$ and $\operatorname{Area}(S_2)$ separately.

First, $\phi^{-1}(S_1)$ is contained in the $\frac{16\sqrt{2}}{1-\rho_2}|\phi'(0)|^{-1}e^{1-s}$ neighbourhood of $\partial W$ inside $W$, by Koebe’s distortion theorem.
Hence we have the simple estimate
\begin{equation}\label{eqn-S}
\operatorname{Area}(\phi^{-1}(S_1))\leq \frac{16\sqrt{2}\mathrm{Length}(W)}{(1-\rho_2)}|\phi'(0)|^{-1}e^{1-s}
\end{equation}
Applying Koebe’s distortion theorem and the condition (c),  we obtain
\begin{equation}\label{eqn-W}
\frac{1}{C'}\leq \frac{\mathrm{Length}^2(\partial W)}{\mathrm{Area}(W)}\leq C'
\end{equation}
for some constant $C'$ depending on $\rho_2$ and $\delta$. Thus,
\begin{align*}
\frac{\operatorname{Area}(S_1)}{\operatorname{Area}(D)}
&\leq \left( \frac{1 + \rho_2}{1 - \rho_2} \right)^8
\frac{\operatorname{Area}( \phi^{-1}(S_1) )}{\operatorname{Area}(W)} \\
&\leq \left( \frac{1 + \rho_2}{1 - \rho_2} \right)^8\frac{\operatorname{Area}( \phi^{-1}(S_1) )}{\operatorname{Area}(D(0, \rho_1))}\\
(\text{by}\ (\ref{eqn-S}))\qquad&\leq \frac{16\sqrt{2} (1 + \rho_2)^8}{\pi \rho_{1}^2(1 - \rho_2)^{9}} |\phi'(0)|^{-1} \exp(1 - s)\mathrm{Length}(\partial W)\\
(\text{by}\ \phi'(0)\geq l/L\ \text{and}\ (\ref{eqn-W}))\qquad&\leq  \frac{16\sqrt{2C'} (1 + \rho_2)^8\rho_2 L}{\sqrt{\pi} \rho_{1}^2(1 - \rho_2)^{9}l}\exp(1 - s).
\end{align*}

Next, we estimate $\operatorname{Area}(S_2)$. Recall that $T = \{T_j\}_{j=1}^N$, where each rectangle $T_j$ has length $\mathrm{diam}(D)$ and width $\frac{\pi}{A}e^{(c_1-1)s} + \frac{2e^2}{A}e^{-s}$.
The set $S_2$ can be covered by $N$ rectangles $\{\hat T_j\}_{j=1}^N$, each of length $\mathrm{diam}(D)+8\sqrt{2}e^{1-s}$ and width $\frac{\pi}{A}e^{(c_1-1)s} + \frac{2e^2}{A}e^{-s}+8\sqrt{2}e^{1-s}$.
For sufficiently large $s$ one easily obtains
\begin{equation*}
\frac{\operatorname{Area}(\hat{T}_j)}{\operatorname{Area}(T_j)}\leq 2
\end{equation*}
for $1\leq j\leq N$. 
Therefore,
$$
\frac{\operatorname{Area}(S_2)}{\operatorname{Area}(T)}=
\frac{\sum_{j=1}^{N}\operatorname{Area}(\hat{T}_j)}
{\sum_{j=1}^{N}\operatorname{Area}(T_j)}\leq \max_{1\leq j\leq N}\frac{\operatorname{Area}(\hat{T}_j)}{\operatorname{Area}(T_j)}\leq 2. 
$$
Putting the estimates together, we finally obtain
\begin{align*}
\text{Area}\left(\mathcal{T}_{\mu, s}^{\text{grid}}\right)
\leq& \frac{16\sqrt{2C'} (1 + \rho_2)^8\rho_2 L}{\sqrt{\pi} \rho_{1}^2(1 - \rho_2)^{9}l}\exp(1 - s)\mathrm{Area}(D)
\\
+&2C(\exp(-s)+\exp(-(c_1-1)s))\mathrm{Area}(D).
\end{align*}

\end{proof}

\section{Proof of Theorem \ref{thm-key}}\label{sec:key}
In this section, we prove the key result (Theorem \ref{thm-key}). We begin with the following elementary estimate.
\begin{lemma}\label{lem61}
For any constant \( C > 0 \), there exists \( \alpha_0 > 0 \) such that for any \( \alpha > \alpha_0 \), \( k > 0 \), and \( \lambda = (\lambda_n)_{n=1}^{\infty} \in [A, B]^{\mathbb{N}} \), one can choose a maximal integer \( m \) satisfying
\begin{equation}\label{eqn-m}
e^{-\alpha} \prod_{i=1}^{m-1} f_{\sigma^k \lambda}^i(0) \leq C
\end{equation}
and
\[
\prod_{i=1}^{j-1} f_{\sigma^k \lambda}^i(0) \leq \prod_{i=1}^{m-2} f_{\sigma^k \lambda}^i(0) \quad \text{for all }\ 2 \leq j \leq m-1.
\]
\end{lemma}
\begin{proof}
First observe that there exists $s_0>0$ such that
$$f_{\sigma^k\lambda}^s(0)\geq 2\
\text{for all}\ s\geq s_0.$$ We now verify the second inequality in two cases.

\noindent \textbf{Case 1: } $j\geq s_0$. 

For any $\ell\geq j+2$, we have
\begin{align*}
\frac{\prod_{i=1}^{\ell-2} f_{\sigma^k\lambda}^i(0)}{\prod_{i=1}^{j-1} f_{\sigma^k\lambda}^i(0)}=\prod_{i=j}^{\ell-2} f_{\sigma^k\lambda}^i(0)\geq 1.
\end{align*}

\noindent \textbf{Case 2: } $j<s_0$. 

Let \( m > 0 \) be a sufficiently large integer, to be determined later. Then
\begin{align*}
\frac{\prod_{i=1}^{m-2} f_{\sigma^k \lambda}^i(0)}{\prod_{i=1}^{j-1} f_{\sigma^k \lambda}^i(0)} 
&= \prod_{i=j}^{s_0 - 1} f_{\sigma^k \lambda}^i(0) \cdot \prod_{i=s_0}^{m-2} f_{\sigma^k \lambda}^i(0) \\
&\geq \left( \prod_{i=j}^{s_0 - 1} f_{A^\N}^i(0) \right) \left( \prod_{i=s_0}^{m-2} f_{A^\N}^i(0) \right).
\end{align*}where $A^{\N}:=(A, A, \cdots, A, \cdots).$

Let
\[
M = \min \left\{ \prod_{i=j}^{s_0 - 1} f_{A^\N}^i(0) : 1 \leq j \leq s_0 - 1 \right\}.
\]
Then the product above is at least
$M \cdot \displaystyle\prod_{i=s_0}^{m-2} f_{A^\mathbb{N}}^i(0).$
Because \( f_{A^\N}^i(0) \to \infty \) as \( i \to \infty \), we can pick \( m_0 > s_0 \) such that for every \( m \geq m_0 \),
\[
\prod_{i=s_0}^{m-2} f_{A^\N}^i(0) \geq \frac{1}{M}.
\]
Therefore, for all \( m \geq m_0 \) and \( 2 \leq j \leq m-1 \), we have
\[
\prod_{i=1}^{j-1} f_{\sigma^k \lambda}^i(0) \leq \prod_{i=1}^{m-2} f_{\sigma^k \lambda}^i(0).
\]

It remains to show that the maximal \( m \) satisfying the first inequality (\ref{eqn-m}) is not smaller than \( m_0 \). Choose \( \alpha_0 > 0 \) such that
\begin{equation} \label{eq-inequality}
e^{-\alpha_0} \prod_{i=1}^{m_0 - 1} f_{B^\N}^i(0) < C.
\end{equation}
Suppose, for contradiction, that for some \( \alpha > \alpha_0 \), the maximal such \( m \) is less than \( m_0 \). Then
\[
e^{-\alpha} \prod_{i=1}^{m_0 - 1} f_{\sigma^k \lambda}^i(0) > C.
\]
Since \( f_{\sigma^k \lambda}^i(0) \leq f_{B^\N}^i(0) \) for all \( i \), we have
\[
e^{-\alpha_0} \prod_{i=1}^{m_0 - 1} f_{B^\N}^i(0) 
\geq e^{-\alpha} \prod_{i=1}^{m_0 - 1} f_{\sigma^k \lambda}^i(0) > C,
\]
contradicting \eqref{eq-inequality}. Hence, the maximal \( m \) is at least \( m_0 \), and the lemma is proved.

\end{proof}

Let $K$ be the constant given by Koebe's distortion theorem such that for any $\delta>0$, if $\varphi$ is a conformal map of a square $S(z_0, 2\delta)$ into $\C$, then $$\frac{1}{K}\le\frac{|\varphi'(z)|}{|\varphi'(w)|}\le K\quad\text{for}\ z, w\in S(z_0, 3\delta/2).$$

The proof of Theorem \ref{thm-key} follows an argument similar to that of Rees \cite{Rees86}. We therefore also need to prove the following lemma, which corresponds to \cite[Lemma 1]{Rees86}.
\begin{lemma}\label{lem-Ree}
    Let \( S = S(w_0, 1) \) with \( \Re(w_0) = -\alpha \) and \(\alpha \geq \alpha_0\). For sufficiently large \(\alpha_0\) and any $k\in \mathbb{N}$, define $m:=m(\alpha,k)$ as the integer satisfying 
    \[
    m=\max\left\{n\in\N:e^{-\alpha+2} \prod_{i=1}^{n-1} f_{\sigma^{k+1}\lambda}^i(0) \leq \frac{1}{BeK}\right\}.
    \]
    Then \( f_{\sigma^k \lambda}^{m+1}\bigg|_S \) is conformal and the following hold:
    \begin{enumerate}
        \item \( \dfrac{1}{K} \leq \dfrac{|( f_{\sigma^k \lambda}^{m+1})'(z)|}{|( f_{\sigma^k \lambda}^{m+1})'(w)|} \leq K \) for all \( z \), \( w \in S \), and
        \[
         f_{\sigma^k \lambda}^{m+1}(S) \subseteq K|( f_{\lambda_{k+m+1}})'(w)| S\left(f_{\sigma^k \lambda}^{m}(0), \frac{1}{2}f_{\sigma^k \lambda}^{m}(0) \right),
        \]
        \item  \( d(f_{\sigma^{k+1} \lambda}^{i}(0), f_{\sigma^k \lambda}^{i+1}(S)) < e^{-\alpha/2} \) if \( i \leq m-2 \), and \( d(f_{\sigma^{k+1} \lambda}^{m-1}(0), f_{\sigma^k \lambda}^{m}(S)) \leq  \frac{1}{e} \),
        \item \( \Re(z) > \alpha/2 \) if \( z \in f_{\sigma^k \lambda}^{m}(S) \), and \( \Re(z) > \exp(\alpha/2) \) if \( z \in f_{\sigma^k \lambda}^{m+1}(S) \),
        \item \( \operatorname{diam}(f_{\sigma^k \lambda}^{m+1}(S)) \geq \dfrac{2A}{Be^3K^2} \).
    \end{enumerate}
\end{lemma}
\begin{proof}
Choose $\alpha_0$ sufficiently large such that for any $\alpha>\alpha_0$, the following property holds:
$$
\frac{\sum_{j=1}^{m-1}2^{2j-3}\cdot 9^{2j-2}}{BeK f_{\sigma^{k+1}\lambda}^{m-1}(0)}\leq 1/2. 
$$
Set $S'=2S=S(w_0, 2)$ and $S''=\frac{3}{2}S=S(w_0, 3/2)$.
Define
$$
C=\min_{\lambda\in [A, B]^\mathbb{N}}\min\left\{\frac{\prod\limits_{i=j}^{m-1}{f_{\sigma^{k+1}\lambda}^{i}(0)}}
{f_{\sigma^{k+1}\lambda}^{m-1}(0)}: 1\leq j\leq m-2\right\}.
$$
For any $t\in[A, B]$,
the map $f_{t}$ is conformal form $S'$ to $f_{t}(S')$, and
$$
f_{t}(S')\subset D(0, Be^{-\alpha+2}),\quad
\text{diam}(f_{t}(S'))\geq 2Ae^{-\alpha}. 
$$
These estimates hold in particular for  $t=\lambda_{k+1}$.

Let $r_1=Be^{-\alpha+2}$. Then, 
$$
r_1\leq \frac{1}{eK\prod_{i=1}^{m-1} f_{\sigma^{k+1}\lambda}^i(0)}\leq \frac{1}{CeK f_{\sigma^{k+1}\lambda}^{m-1}(0)} .
$$
Without loss of generality we may assume that $r_1\leq \frac{1}{2}$. Since $f_{\lambda_{k+2}}$ is injective on $D(0, 1)$, Koebe's distortion theorem gives
\begin{align*}
\operatorname{diam}\left(f_{\sigma^k\lambda}^2(S')\right)
&\leq |f_{\lambda_{k+2}}'(0)| \cdot \frac{(1+r_1)^4}{(1-r_1)^4} \cdot \operatorname{diam}(f_{\lambda_{k+1}}(S')) \\
&\leq |f_{\lambda_{k+2}}'(0)| \cdot \frac{(1+r_1)^4}{(1-r_1)^4} \cdot 2r_1 \\
&\leq 2 \cdot 9^2 |f_{\lambda_{k+2}}'(0)| \cdot r_1.
\end{align*}
Denote
$$
r_2=2^2\cdot 9^2|f_{\lambda_{k+2}}'(0)|r_1=2^2\cdot 9^2|f_{\lambda_{k+2}}(0)|r_1.
$$
Then,
$$
r_2\leq \frac{2^2\cdot 9^2}{eK f_{\sigma^{k+1}\lambda}^{m-1}(0)}\leq \frac{1}{2},
\qquad f_{\sigma^k\lambda}^2(S')\subset D(0, r_2).$$
Proceeding inductively, for \(3\le j\le m-1\), suppose \(r_{j-1}\le\frac12\) has been defined. 
Then
\begin{align*}
\operatorname{diam}\left(f_{\sigma^{k}\lambda}^j(S')\right)
&\leq |f_{\lambda_{k+j}}'(f_{\sigma^{k+1}\lambda}^{j-2}(0))| \cdot \frac{(1+r_{j-1})^4}{(1-r_{j-1})^4} \cdot \operatorname{diam}(f_{\sigma^k\lambda}^{j-1}(S')) \\
&\leq |f_{\lambda_{k+j}}'(f_{\sigma^{k+1}\lambda}^{j-2}(0))| \cdot \frac{(1+r_{j-1})^4}{(1-r_{j-1})^4} \cdot 2r_{j-1} \\
&\leq 2 \cdot 9^2 |f_{\lambda_{k+j}}'(f_{\sigma^{k+1}\lambda}^{j-2}(0))| \cdot r_{j-1}.
\end{align*}
Set
$$r_j=2^2\cdot 9^2|f_{\lambda_{k+j}}'(f_{\sigma^{k+1}\lambda}^{j-2}(0))|r_{j-1}=2^2\cdot 9^2|f_{\sigma^{k+1}\lambda}^{j-1}(0)|r_{j-1},$$ 
which implies
\begin{align*}
   r_j&=2^{2j-2}\cdot 9^{2j-2}\prod_{i=1}^{j-1}|f_{\sigma^{k+1}\lambda}^{i}(0)|r_1\\
   &\le \frac{2^{2j-2}\cdot 9^{2j-2}\prod_{i=1}^{j-1}|f_{\sigma^{k+1}\lambda}^{i}(0)|}{eK\prod_{i=1}^{m-1}|f_{\sigma^{k+1}\lambda}^{i}(0)|}
   = \frac{2^{2j-2}\cdot 9^{2j-2}}{eK\prod_{i=j}^{m-1}|f_{\sigma^{k+1}\lambda}^{i}(0)|}
\end{align*}
Hence, $r_j\leq \frac{1}{2}$ and
$$
\text{diam}(f_{\sigma^k\lambda}^{j}(S'))\leq \frac{r_j}{2}
\leq \frac{2^{2j-3}\cdot 9^{2j-2}}{eK f_{\sigma^{k+1}\lambda}^{m-1}(0)}\leq \frac{1}{2}.
$$
This implies that
 \( d(f_{\sigma^{k+1} \lambda}^{j}(0), f_{\sigma^k \lambda}^{j+1}(S')) < e^{-\alpha/2} \) if \( j \leq m-2. \)

Now, $f_{\lambda_{k+m}}$ is injective on $f_{\sigma^k\lambda}^{m-1}(S')$.  By Koebe's distortion theorem, we have
$$
\mathrm{diam}(f_{\sigma^k\lambda}^{m}(S''))\leq 2KBe^{-\alpha+2}\prod_{j=1}^{m-1}f_{\sigma^{k+1}\lambda}^{j}
(0)\leq\frac{2}{e}<1.
$$
Hence 
\[ d(f_{\sigma^{k+1} \lambda}^{m-1}(0), f_{\sigma^k \lambda}^{m}(S'')) \leq 2. \]
Moreover,
$f_{\lambda_{k+m+1}}$ is injective on $f_{\sigma^k\lambda}^{m}(S'')$. Applying Koebe’s distortion theorem once more gives  

\[ \dfrac{1}{K} \leq \dfrac{|( f_{\sigma^k \lambda}^{m+1})'(z)|}{|( f_{\sigma^k \lambda}^{m+1})'(w)|} \leq K \] for all \( z \), \( w \in S \)
and
\begin{align*}
\operatorname{diam}(f_{\sigma^k \lambda}^{m+1}(S))&\ge\frac{1}{K}|(f_{\sigma^{k+1}\lambda}^m)'(0)|\text{diam}(f_{\lambda_{k+1}}(S)) \\
&\geq
\frac{1}{K}2Ae^{-\alpha}\prod_{j=1}^{m}
f_{\sigma^{k+1}\lambda}^{j}(0)\geq   \dfrac{2A}{Be^3K^2}.
\end{align*}
For the statement (ii), if $i\leq m-2$, from the Koebe distortion theorem, we have
\begin{align*}
    d(f_{\sigma^{k+1} \lambda}^{i}(0), f_{\sigma^k \lambda}^{i+1}(S))&\leq K|(f_{\sigma^{k+1}\lambda}^i(0))'|d(0, f_{\lambda_{k+1}}(S))\\
    &\le KBe^{-\alpha+2}\prod_{j=1}^i|f_{\sigma^{k+1}\lambda}^j(0)|\\
(\text{by Proposition\ \ref{pro1}})    &\le KBe^{2-\alpha}(f_{\sigma^{k+1}\lambda}^{m-2}(0))^{3/2}<e^{-\alpha/2}.
\end{align*}
If $i=m-1$, we have 
\begin{align*}
    d(f_{\sigma^{k+1} \lambda}^{m-1}(0), f_{\sigma^k \lambda}^{m}(S))&\leq K|(f_{\sigma^{k+1}\lambda}^{m-1}(0))'|d(0, f_{\lambda_{k+1}}(S))\\
    &\le KBe^{-\alpha+2}\prod_{j=1}^{m-1}|f_{\sigma^{k+1}\lambda}^j(0)|\le\frac{1}{e}
\end{align*}

The remaining claims in (iii) follow directly from the construction and the chosen bounds.
\end{proof}

\begin{proof}[Proof of Theorem \ref{thm-key}]
Fix a sufficiently large $M$. Define $h(x) = \exp\left(\frac{x}{2}\right)$ and set $M_1:=M$ and $M_n = h^{n-1}(M)$ for $n \in \mathbb{N}$. Then $M_n > M_{n-1}$ and $M_n \to \infty$ as $n \to \infty$.

Let $Q_0^1$ be a square of side length exactly $2$ in $\mathcal{H}_M$, and define $$\mathcal{F}_0 = \mathcal{G}_0 = {Q_0^1},\quad X_0 = Q_0^1.$$

We now define a sequence of families $\{\mathcal{F}_k\}_{k=1}^\infty$ recursively. starting with the construction of $\mathcal{F}_1$. 
By applying Proposition \ref{pro-type0} and \ref{pro-grid} to $(D=Q_0^1, D'=2Q_0^1)$, we obtain a finite collection of squares  $\{Q_{01}^{1k_1}\}_{k_1\in I_{0}^{1}}$ with side-length $2$, contained in $f_{\lambda_1}(D)\subset\{z: |\Re z| >M_2\}$ such that
$$
\frac{\sum\limits_{k_1\in I_{0}^{1}}
\mathrm{Area}(Q_{01}^{1k_1})}{\mathrm{Area}(
f_{\lambda_1}(Q_{0}^{1}))}\geq 1- \frac{c}{M_1}\ \text{for some constant}\ c>0.
$$
Here $I_0^1$ is a finite index set.

Define $$I_0^{1,+}:=\{k_1\in I_0^1: Q_{01}^{1k_1}\subset\{\Re z>M_2\}\}$$ and 
$$I_0^{1,-}:=\{k_1\in I_0^1: Q_{01}^{1k_1}\subset\{\Re z<-M_2\}\}.$$
Clearly, $I_0^1=I_0^{1,-}\cup I_0^{1,+}.$

For each $k_1\in I_{0}^{1,-}$, from Lemma \ref{lem-Ree}, there exists $l_1^{k_1}>2$ such that
$f_{\sigma \lambda}^{l_1^{k_1}}(Q_{01}^{1k_1})$ is a Jordan domain with $L$-bounded shape whose diameter is bounded below by an absolute constant $l>0$. Moreover,
$$f_{\sigma \lambda}^{l_1^{k_1}-1}(Q_{01}^{1k_1})\subset \{z: \Re z>M_2\}.$$
In this case, using Proposition \ref{pro-type0} and \ref{pro-grid} to the pair $(D:=f_{\sigma \lambda}^{l_1^{k_1}-1}(Q_{01}^{1k_1}), D':=f_{\sigma \lambda}^{l_1^{k_1}-1}(2Q_{01}^{1k_1}))$, we obtain that
$f_{\sigma \lambda}^{l_1^{k_1}}(Q_{01}^{1k_1})$ also contains finitely many squares
$\left\{Q_{01l_1^{k_1}}^{1k_1k_2}\right\}_{k_2\in I_{01}^{1k_1}}$ with side-length 2 satisfying
$$
\frac{
\sum\limits_{k_2\in I_{01}^{1k_1}}
\mathrm{Area}\left(Q_{01l_{1}^{k_1}}^{1k_1k_2}\right)}{\mathrm{Area}\left(
f_{\sigma\lambda}^{l_1^{k_1}}(Q_{01}^{1k_1})\right)}\geq 1-\frac{c}{M_2}.
$$ where $\{I_{01}^{1k_1}\}$ is also a finite index.

Define $$\mathcal{F}_1=\left\{Q_{01l_{1}^{k_1}}^{1k_1k_2}: k_1\in I_{0}^{1,-}, k_2\in I_{01}^{1k_1}\right\}\cup\{Q_{01}^{1k_1}: k_1\in I_{0}^{1,+}\}$$
and
$$
\mathcal{G}_1=\left\{Q_{0}^{1}\cap \left(f_{\lambda}^{l_{1}^{k_1}+1}\right)^{-1}Q_{01l_1^{k_1}}^{1k_1k_2}: k_1\in I_{0}^{1,-}, k_2\in I_{01}^{1k_1}\right\}\cup\left\{Q_0^1\cap f_{\lambda_1}^{-1}
(Q_{01}^{1k_1}): k_1\in I_{0}^{1,+}\right\}
$$ 
For any $G\in \mathcal{G}_1$, we have
$$
\inf_{z\in G}|z|\geq M_1.
$$
By Koebe's distortion theorem, we have
$$
\frac{
\sum\limits_{k_1\in I_{0}^{1,-}}\sum\limits_{k_2\in I_{01}^{1k}}
\mathrm{Area}\left(Q_{0}^{1}\cap \left(f_{\lambda}^{l_1^{k_1}+1}\right)^{-1}
Q_{01l_1^{k_1}}^{1k_1k_2}\right)}{
\sum\limits_{k_1\in I_{0}^{1, +}}\mathrm{Area}\left(Q_{0}^{1}\cap
f_{\lambda_1}^{-1}
(Q_{01}^{1k_1})\right)}\geq 1-\frac{c^2}{M_2}.
$$
Set $X_1=\bigcup\limits_{U\in \mathcal{G}_1} U$. 
Then,
\begin{align*}
\mathrm{Area}(X_1)&=\sum\limits_{U\in \mathcal{G}_1}\mathrm{Area}(U)\\
&=\sum\limits_{k_1\in I_{0}^{1,-}}\sum\limits_{k_2\in I_{01}^{1k}}
\mathrm{Area}\left(Q_{0}^{1}\cap \left(f_{\lambda}^{l_1^{k_1}+1}\right)^{-1}
Q_{01l_1^{k_1}}^{1k_1k_2}\right)+\sum\limits_{k_1\in I_{0}^{1, +}}\mathrm{Area}\left(Q_{0}^{1}\cap
f_{\lambda_1}^{-1}
(Q_{0,1}^{1k_1})\right)\\
&\geq (1-\frac{c^2}{M_2})\sum_{k_1\in I_{0}^{1,+}}
\mathrm{Area}\left(Q_{0}^{1}\cap
f_{\lambda_1}^{-1}
(Q_{0,1}^{1k_1})\right)\geq (1-\frac{c^2}{M_2})(1-\frac{c}{M_1})\mathrm{Area}(Q_{0}^{1})
\\
&\geq (1-\frac{2c}{M_1})\mathrm{Area}(Q_{0}^{1})=(1-\frac{2c}{M_1})\mathrm{Area}(X_0).
\end{align*}
Proceeding inductively, we can construct $\mathcal{F}_{m}, \mathcal{G}_{m}, X_m$ and finite index sets $I_{01l_{1}^{k_1}\dots l_{m-1}^{k_{m-1}}}^{1k_1k_2\dots k_m}$ for $m\leq s-1$ such that
$X_0\supset X_1\supset \dots \supset X_{s-1}$ and 
$$
\mathrm{Area}(X_{s-1})\geq (1-\frac{2c}{M_{s-1}})\mathrm{Area}(X_{s-2}).
$$

 We now illustrate the construction of $\mathcal{F}_{s}, \mathcal{G}_{s}, X_{s}$ in the following.
 
For any $Q\in \mathcal{F}_{s-1}$, actually, we can write 
$Q=Q_{01l_1^{k_1}\cdots l_{s-1}^{k_{s-1}}}^{1k_1k_2\cdots k_{s}}$ and note that
$$
\inf_{z\in Q}|\Re z|\geq M_s.
$$
\textbf{Case 1:} Suppose $k_s\in I_{01l_{1}^{k_1}\dots l_{s-2}^{k_{s-2}}}^{1k_1k_2\dots k_{s-1}, -}$, i.e., $Q\subset\{z: \Re z<-M_s\}$. Then, there exists an integer $l_{s}^{k_s}>2$ such that
$f_{\sigma^{t_s}\lambda}^{l_{s}^{k_s}}(Q)$ is a Jordan domain of $L$-bounded shape with diameter bounded below by an absolute constant $l>0$, where $t_s=1+\sum\limits_{j=1}^{s-1}l_{j}^{k_j}$. Moreover,
$$f_{\sigma^{t_s}\lambda}^{l_{s}^{k_s}}(Q)\subset \{z: \Re z>M_{s+1}\}$$ and it contains finitely many squares of side-length 2,
$$\left\{Q_{01l_{1}^{k_1}\dots l_{s-1}^{k_{s-1}}l_{s}^{k_s}}^{1k_1k_2\dots k_sk_{s+1}}:
k_{s+1}\in I_{01l_{1}^{k_1}\dots l_{s-1}^{k_{s-1}}}^{1k_1k_2\dots k_s}
\right\}$$
where $I_{01l_{1}^{k_1}\dots l_{s-1}^{k_{s-1}}}^{1k_1k_2\dots k_s}$ is a finite index set. 

\noindent\textbf{Case 2:} Suppose $k_s\in I_{01l_{1}^{k_1}\dots l_{s-2}^{k_{s-2}}}^{1k_1k_2\dots k_{s-1}, +}$, i.e., $Q\subset\{z:\Re z>M_s\}$. Then $f_{\sigma^{t_s}\lambda}^{1}(Q)=f_{\lambda_{t_s+1}}(Q)$ contains finitely many squares of side-length $2$, 
$$\left\{Q_{01l_{1}^{k_1}\dots l_{s-1}^{k_{s-1}}1}^{1k_1k_2\dots k_sk_{s+1}}:
k_{s+1}\in I_{01l_{1}^{k_1}\dots l_{s-1}^{k_{s-1}}}^{1k_1k_2\dots k_s}
\right\}.
$$
Now, define
$$\mathcal{F}_{s+1, -}=
\left\{Q_{01l_{1}^{k_1}\dots l_{s-1}^{k_{s-1}}l_{s}^{k_s}}^{1k_1k_2\dots k_sk_{s+1}}: 
Q_{01l_{1}^{k_1}\dots l_{s-1}^{k_{s-1}}}^{1k_1k_2\dots k_s}\in \mathcal{F}_n, k_s\in I_{01l_{1}^{k_1}\dots l_{s-2}^{k_{s-2}}}^{1k_1k_2\dots k_{s-1}, -}, k_{s+1}\in I_{01l_{1}^{k_1}\dots l_{s-1}^{k_{s-1}}}^{1k_1k_2\dots k_s}
\right\},$$
$$
\mathcal{F}_{s+1, +}=
\left\{Q_{01l_{1}^{k_1}\dots l_{s-1}^{k_{s-1}}1}^{1k_1k_2\dots k_sk_{s+1}}: 
Q_{01l_{1}^{k_1}\dots l_{s-1}^{k_{s-1}}}^{1k_1k_2\dots k_s}\in \mathcal{F}_n, k_s\in I_{01l_{1}^{k_1}\dots l_{s-2}^{k_{s-2}}}^{1k_1k_2\dots k_{s-1}, +}, k_{s+1}\in I_{01l_{1}^{k_1}\dots l_{s-1}^{k_{s-1}}}^{1k_1k_2\dots k_s}
\right\},
$$
and set
$$\mathcal{F}_{s+1}=\mathcal{F}_{s+1, -}\cup\mathcal{F}_{s+1, +}.$$  
Define further
$$
\mathcal{G}_{s+1, -}=\left\{Q_{0}^{1}\cap (f_{\lambda}^{t_s+l_s^{k_s}})^{-1}(U): U\in \mathcal{F}_{s+1, -}\right\},
$$
$$
\mathcal{G}_{s+1, +}=\left\{Q_{0}^{1}\cap (f_{\lambda}^{t_s+1})^{-1}(U): U\in \mathcal{F}_{s+1, +}\right\}
$$
and set
$$
\mathcal{G}_{s+1}=\mathcal{G}_{s+1, -}\cup\mathcal{G}_{s+1, +}
$$
Let $X_{s+1}=\bigcup\limits_{V\in \mathcal{G}_{s+1}}V$. As in the previous step, we obtain
$$
\mathrm{Area}(X_{s+1})\geq (1-\frac{2c}{M_{s+1}})
\mathrm{Area}(X_{s}). 
$$
Finally, write $X_{\infty}=\displaystyle\bigcap_{s=1}^{\infty}X_{s}$. Then,
$$\mathrm{Area}(X_{\infty})>0.$$
\end{proof}

\section{Appendix}
In this section, we give some elementary inequality of exponential functions.
\begin{lemma}\label{lemma-ratio}
Let $1/e<A<B<\infty$. 
For any $t\in[A, B]$ and $x\in (0, +\infty)$, one has
\begin{equation}\label{eq-one-step}
f_t(x)\ge Aex.
\end{equation}
Moreover,
for any sequence $\lambda\in[A, B]^{\N}$ and $n\in\N$, we have
$$f_{\lambda}^n(x)\geq f_{A^{\N}}^n(0)=f_{A^{\N}}^{n-1}(A),$$ and 
$$
\lim_{n\to\infty}f_{\lambda}^n(x)=\infty, 
$$
for every $x\in(0,\infty)$,
where $A^{\N}:=(A, A, \cdots, A, \cdots)$.
\end{lemma}
\begin{proof}
Let $t\ge A$. Define $g(x)=f_{t}(x)/x, x>0$. It is not hard to see that $g$ decreases in $(0, 1)$ and increases in $(1, \infty)$. Hence
$$
g(x)\geq g(1)=te\geq Ae,\ \text{for all}\ x>0.
$$ and then we obtain (\ref{eq-one-step}).

Since $\lambda\in[A, B]^{\N}$, it follows that
$$
f_{\lambda}^n(x)\geq f_{A^{\mathbb{N}}}^n(x)\geq f_{A^{\mathbb{N}}}^n(0)=
f_{A^{\mathbb{N}}}^{n-1}(A).
$$
By \eqref{eq-one-step}, we obtain
$$
f_{A^{\mathbb{N}}}^{n-1}(A)\geq A(Ae)^{n-1},
$$
which implies the second part.
\end{proof}

As a result, we derive the following proposition immediately. 
\begin{proposition}\label{pro1}
Let $1/e<A<B<\infty$.
For any $\epsilon>0$, there exists $m_0=m_0(\epsilon, A, B)$ such that for any $m\ge j\ge m_0$ and $\mu=(\mu_n)_{n=1}^{\infty}$, we have
$$
f_{\mu}^{j}(0)\leq (f_{\mu}^{m+1}(0))^{\epsilon}.
$$
\end{proposition}
\begin{proof}
Since $f_{\mu}^{m}(0)$ converges to infinity uniformly in $[A, B]^{\mathbb{N}}$, there exists an integer $m_1:=m_1(A, B)$ such that $f_{\mu}^{m_1}(0)>1$ for all $\mu$. Hence, for any $j\geq m_1$, by the fact that $\frac{1}{e}e^x>x$ for $x>1$, we have 
$$f_{\mu}^j(0)\leq f_{\mu}^m(0).$$
Note that for fixed $\epsilon>0$, the inequality $$A^{\epsilon}\exp(\epsilon x)\geq x$$ holds for $x$ sufficiently large, it follows that there exists an integer $m_2:=m_2(\epsilon)$ such that for $m\ge m_2$, $$f_{\mu}^m(0)\leq (f_{\mu}^{m+1}(0))^{\epsilon}.$$
Taking $m_0=\max(m_1, m_2)$, we complete the proof.
\end{proof}

The following lemma is obvious. 
\begin{lemma}\label{lem-mapping}
Let $\lambda=(\lambda_n)_{n=1}^{\infty}\in [A, B]^{\mathbb{N}}$ with $0<A<B$. Then
$$
\mathrm{Sing}((f_{\lambda}^n)^{-1})=\{f_{\sigma^k\lambda}^{n-k}(0): 1\leq k\leq n-1\}\cup\{0\}. 
$$
\end{lemma}

\bibliographystyle{amsalpha}
  \bibliography{bibfile}
\end{document}